\documentclass[11pt]{article}
\usepackage{amsmath, amsthm, amssymb}
\usepackage{amsmath}
\usepackage{amsthm}
\usepackage{amscd,amssymb}
\usepackage[all]{xy}
\usepackage{color}
\usepackage{comment}
\usepackage{appendix}
\usepackage{verbatim}
\usepackage{mathrsfs}
\usepackage{pdfsync}
\usepackage{tikz}
\usetikzlibrary{shapes,arrows}
\usepackage{caption}
\usepackage{subcaption}
\usetikzlibrary {positioning}
\usepackage{indentfirst}
\usepackage[T1]{fontenc}
\usepackage[ansinew]{inputenc}
\usepackage{float}
\usetikzlibrary{shapes,arrows}
\usepackage{setspace}
\usepackage{lipsum}
\date{\today}
\usepackage[colorlinks]{hyperref}
\usepackage{graphicx,epsfig}
\newtheorem{thm}{Theorem}[section]

\newtheorem{lem}[thm]{Lemma}

\let\oldproofname=\proofname
\renewcommand{\proofname}{\rm\bf{\oldproofname}}

 \newtheorem{prop}[thm]{Proposition}
  
\theoremstyle{definition}
 
 \newtheorem{rmk}[thm]{Remark}
 \newtheorem{defn}{Definition}[section]
 \newtheorem{claim}[thm]{Claim}

\newcommand{\A}{\mathcal{A}}

\newcommand{\M}{\mathcal{M}}

\newcommand{\si}{\sigma}
\newcommand{\nin}{\notin}
\newcommand{\sm}{\setminus}
\newcommand{\s}{\mathcal{S}}

\numberwithin{equation}{section}

\begin{document}
\title{Homotopy type of  the independence complexes of a family of regular bipartite graphs }
\author{ Nandini Nilakantan\footnote{Department of Mathematics and Statistics, IIT Kanpur, Kanpur-208016, India. nandini@iitk.ac.in.},  Samir Shukla\footnote{{Department of Mathematics and Statistics, IIT Kanpur, Kanpur-208016, India. samirs@iitk.ac.in.}}}
\maketitle
\begin{abstract}
In this article, we  define a family of regular bipartite graphs and show that the homotopy type of the independence complexes of
this family is the wedge sum of spheres of certain dimensions.
\end{abstract}

 \section{Introduction}
 
The {\it independence complex} $\text{Ind}(G)$ of a graph $G$ is the simplicial complex, whose simplices are 
the independent sets of $G$.  
 Babson and Kozlov used the topology of independence complexes of cycles in the proof of the Lov\'{a}sz conjecture \cite{BK1}.
 Given a simplicial complex $X$, in \cite{eh}, Ehrenborg and Hetyei proved  that there exists a graph $G$, such
that $X$ is homeomorphic to Ind$(G)$. 
 In \cite{jj}, Jonsson proved that the independence complexes
of bipartite graphs have the same homotopy type as those of the suspensions of simplicial complexes.

Let $m$ and $d$ be positive integers with $d \leq m$. 
 The $d$-regular bipartite graph $G_m^d$ is defined to be the graph, whose set of vertices 
$V(G_m^d) = A \sqcup B$. Here, $A= \{a_1, \ldots, a_m\}$  and $ B= \{b_1, \ldots , b_m\}$.  For each $a_i \in A$, define
the neighbors of $a_i$, $N(a_i)$, to be the set
$\{b_i, b_{i+1  (\text{mod} \ m)}, \ldots, b_{i+ d-1  \text{(mod} \ m)}\}$. 
Clearly, for any $b_i \in B$, the neighborhood $N(b_i) = \{a_i, a_{i-1(\text{mod} \ m)}, \ldots, a_{i-d+1(\text{mod} \ m)} \}$. 
It is easy to see that for $d = 2$, $G_m^d \cong C_{2m}$, the cycle of length $2m$.
Kozlov proved the following in \cite{DK}.
\begin{thm}  For $r \geq 3$, let $C_r$ denote the cycle of length $r$. Then 
\begin{center}
            $\text{Ind}(C_r) \simeq \begin{cases}
            S^{k-1} \bigvee S^{k-1} & \text{if $r = 3k$},\\
             S^{k-1} & \text{if $ r = 3k \pm 1 $}.\\
                       \end{cases}$
  \end{center}
 \end{thm}
 
 For $d = m-1$, $G_m^d$ is isomorphic to the categorical product of the complete graphs $K_2$ and $K_m$, {\it i.e}, $G_m^d \cong K_2 \times K_m$. The main result of this article is
 \begin{thm} \label{maintheorem} Let $m = t(d+1)+ \alpha$, where $d \geq 1, t \geq 0$ and $0 \leq \alpha \leq d$. Then 
\begin{center}
   $\text{Ind} (G_m^d) \simeq \begin{cases}
            \ \bigvee\limits_{\text{d-copies}}S^{2t-1} & \text{if} \ \ \alpha = 0,\\
            \ \ \ S^{2t} &   \text{otherwise}.\\
                            
                       \end{cases}$
 \end{center}
\end{thm}

%
%
%

 \section{Preliminaries}

A  {\it graph} $G$ is a  pair $(V(G), E(G))$,  where $V(G)$  is the set of vertices of $G$  and $E(G) \subset V(G)\times V(G)$ denotes the set of edges.
 If $(x, y) \in E(G)$, it is also denoted by $x \sim y$. Here, $x$ is said to be adjacent to $y$. A non empty subset $\mathcal{I}$ of $V(G)$ is called an {\it independent set}, if for any $x, y \in \mathcal{I}$, $x \nsim y$ in $G$.
A {\it  bipartite graph} is a graph $G$ with subsets  $X$ and $Y$ of $V(G)$ such that $V(G)=X   \sqcup Y$ and $(v,w) \notin E(G)$ if $\{v,w\} \subseteq X$ or $\{v,w\} \subseteq Y$.

For any subset $A\subset V(G)$, the {\it
neighborhood  of A} $N(A)$, is defined as $N(A) =\cup_{a \in A} N(v)$, where the {\it
neighborhood  of $v$}, $N(v)=\{ w \in V(G) \ |  \
(v,w) \in E(G)\}$. 
The {\it degree} of a vertex $v$ is defined  as $\text{deg}(v) =
|N(v)|$. Here $| X|$ represents the
cardinality of the set $X$. A graph $G$ is said to be $d$-regular if $\text{deg}(v) = d$ for all $v \in V(G).$

 A {\it graph homomorphism} from $G$ to $H$ is a function $\phi: V(G) \to V(H)$, where  $ v \sim w$ in $G$ implies that $\phi(v) \sim \phi(w)$ in $H$.
If $\phi$ is bijective and $\phi^{-1} : V(H) \to V(G)$ is also a graph homomorphism, then  $\phi$ is called an {\it isomorphism}. In this case, $G$ is said to be isomorphic to $H$ and is denoted by $G \cong H$.
The {\it categorical product} of two graphs $G$ and $H$, denoted by $G\times H$ is the graph where $V(G\times H)=V(G)\times V(H)$ and  $(g,h) \sim (g',h')$ in $G\times H$  if $g \sim g'$  and $h \sim h'$ in $G$ and $H$ respectively.

A {\it finite abstract simplicial complex X} is a collection of
finite sets such that if $\tau \in X$ and $\sigma \subset \tau$,
then $\sigma \in X$.  
The elements  of $X$ are called {\it simplices}
of $X$. By convention the empty set is a simplex of every abstract simplicial complex. 
If $\sigma \in X$ and $|\sigma |=k+1$, then $\sigma$ is said
to be $k-dimensional$. 
  Let $(X_1, x_1), \ldots, (X_n, x_n)$ be pointed topological spaces and $X$ be the disjoint union $\bigsqcup\limits_{i=1}^{n} X_i$. The quotient space obtained from $X$, by identifying  $x_1, \ldots, x_n$ to a single point $x_0$ is called the {\it wedge sum} of $X_{i}$, $1\leq i \leq n$ and is denoted by $\bigvee\limits_{i=1}^{n} X_i$. 
 

We now introduce some tools from Discrete Morse Theory  which have been
used in this article. R. Forman in \cite{f} introduced, what has now
become a standard tool in Topological Combinatorics, Discrete Morse
Theory. The principal idea of Discrete Morse Theory (simplicial) is to reduce the original
complex  by a sequence of collapses to a homotopy equivalent complex, which is not necessarily
simplicial, but, which has fewer cells (refer \cite{jj1}, \cite{dk}).
\begin{defn}
A {\it partial matching} in a poset $P$ is a subset $\M$ of $P \times P$ such that
\begin{itemize}
\item $(a,b) \in \M$ implies $b\succ a$, {\it i.e. $a<b$ and $\not\exists \,c$ such that $a<c<b$}.
\item each element  in $P$ belongs to at most one element of $\M$.
\end{itemize}
\end{defn}
If $\M$ is a {\it partial matching} on a poset $P$
then, there exists  $A \subset P$ and an injective map $f: A
\rightarrow P\setminus A$ such that $f(x)\succ x$ for all $x \in A$.

\begin{defn}
An {\it acyclic matching} is a partial matching  $\M$ on the poset $P$ such that there does not exist a cycle
\begin{eqnarray*}
f(x_1)  \succ x_1 \prec f( x_2) \succ x_2  \prec f( x_3) \succ x_3 \dots   f(x_t) \succ x_t  \prec f(x_1), t\geq 2.
\end{eqnarray*}
\end{defn}
For an acyclic partial matching on $P$, those elements of $P$ which do not belong to the matching are said to be 
{\it critical }. To obtain the  desired homotopy equivalence, the following result is used.

\begin{thm} (Main theorem of Discrete Morse Theory)\label{dmt1} \cite{f}

Let $X$ be a simplicial complex and $\A$ be an acyclic matching on the face poset of $X$ 
such that the empty set is not critical. 
Then, $X$ is homotopy equivalent to a cell complex which has a $d$ -dimensional cell for each $d$ -dimensional
critical face of $X$ together with an additional $0$-cell.
\end{thm}

 \section{Main Result}

Let    $I$ denote $\text{Ind}(G_m^{d})$,
$P$ the face poset of $(I, \subseteq)$ and $S_{a_1}= \{\sigma \in P \ | \ a_1 \notin \sigma , \sigma \cup  \{a_1\} \in P \}$. 
Define the map $\mu_{a_1} : S_{a_1}  \rightarrow P \setminus S_{a_1}$ by $\mu_{a_1}(\sigma) = \sigma \cup \{a_1\}.$
Let $S_{a_1}' = P \setminus \{S_{a_1}, \mu_{a_1}(S_{a_1})\}.$ For $i \in \{2 , \ldots, m\}$ define $\mu_{a_i} : S_{a_i} 
\rightarrow S_{a_{i-1}}' \setminus S_{a_i}$ to be the map $\mu_{a_i}(\sigma) = \sigma \cup \{a_i\}$, where 
$S_{a_i} = \{\sigma \in S_{a_{i-1}}' \ | \  a_i  \notin \sigma , \sigma \cup \{a_i\} \in S_{a_{i-1}}'\}$ and 
$S_{a_i}'= 
S_{a_{i-1}}' \setminus \{S_{a_i}, \mu_{a_i}(S_{a_i})\}.$  By the above construction, $S_{a_i} \cap S_{a_{j}} =\emptyset$ $\forall \ i \neq j$.

Let $S_1 = \bigcup \limits_{i=1}^{m} S_{a_i}$ and $\mu_1 : S_1 \rightarrow P \setminus S_1$ be defined by $\mu_1(\sigma) =\mu_{a_i} (\sigma)$, if $\sigma \in S_{a_i}$.  

On the set $M = P \setminus \{S_1, \mu_1(S_1)\}$, in a similar manner define the sets  $S_{b_i}$ and $\mu_{b_i}(S_{b_i})$ for all $1 \leq i \leq m$. Let $S_2 = \bigcup \limits_{i=1}^m S_{b_i}$ and the map $\mu_2 : S_2 \rightarrow M \setminus S_2$ be defined by $\mu_2(\sigma) = \mu_{b_i}(\sigma)$, where $i$ is the unique integer such that $\sigma \in S_{b_i}$. Clearly $S_1 \cap S_2 = \emptyset.$ 

Define $\mu: S = S_1 \cup S_2 \rightarrow P \setminus S$ by

\begin{center}
   $\mu(\sigma)= \begin{cases}
            \ \mu_1(\sigma) & \text{if} \ \sigma \in S_1\\
            \ \mu_2(\sigma) & \text{if} \ \sigma \in S_2.\\
                            
                       \end{cases}$
 \end{center}
Clearly, $\mu$ is  injective  and is therefore, a well defined partial matching on $P$.  
 \begin{prop}
 $\mu$ is an acyclic matching on $P$.
  \end{prop}
\begin{proof}
Let there exist distinct cells $\sigma_1, \ldots, \sigma_t \in S$ such that $\mu(\sigma_i) \succ 
\sigma_{i+1 \ (\text{mod} \ t)}, 1 \leq i \leq t$.

The elements of $V(G_m^d)$ are ordered as  $a_1 < a_2 < \ldots < a_m < b_1 < \ldots < b_m.$ 
 Let $x \in V(G_m^d)$ be the least element such that $\{\sigma_1, \ldots, \sigma_t\} \cap S_x \neq \emptyset$.
 Without loss of generality assume that $\sigma_1 \in S_x$ {\it i.e.}, $x \notin \sigma_1$ and $\mu(\sigma_1) = \sigma_{1} \cup \{x\}$. $\mu(\sigma_1) \succ \sigma_2$ and  $\sigma_1 \neq \sigma_2$ implies that there exists $x' \in \mu_1(\sigma_1), x' \neq x$ such that 
 $\sigma_2 = \mu(\sigma_1) \setminus \{x'\}$. We now have the following two possibilities:
 
 \begin{enumerate}
 \item $x \in \si_{t}$. 
 
 $\sigma_1 \in S_x$ implies that $x \notin \sigma_1$.  $x \in \sigma_t$ implies that $x \in \mu(\sigma_t)$. Therefore, $\sigma_1 = \mu(\sigma_t) \setminus \{x\}$ which implies that $\mu(\sigma_1) = \mu(\sigma_t)$ a contradiction, since $\sigma_1 \neq \sigma_t$.
 
 \item $x \nin \si_{t}$, {\it i.e.} $\exists$ a smallest $l \in \{2, \ldots, t\}$ such that $x \notin \sigma_l$.
 
 $x \in \mu(\sigma_{l-1})$ and $x \notin \sigma_l$ implies that $\sigma_l = \mu(\sigma_{l-1}) \setminus \{x\}$ {\it i.e.}, $\mu(\sigma_{l-1}) = \sigma_l \cup \{x\}$. Since $\sigma_l$ and $\mu(\sigma_{l-1}) \notin S_i \cup \mu_i(S_i) \ \forall \ i < x$, from the definition $\sigma_l \in S_x$. This implies that $\mu(\sigma_l) = \sigma_l \cup \{x\} = \mu(\sigma_{l-1})$, which implies that $\sigma_l = \sigma_{l-1}$, a contradiction.

Therefore,  $\mu$ is an acyclic matching on $P$.
\end{enumerate}
 \end{proof}
 
We can now conclude that the set of critical cells corresponding to this matching is $P \sm\{ S, \mu(S)\}$, denoted by $C$.
 Let $M = P \setminus \{S_1 \cup \mu_1(S_1)\}$, $\s_{a_{j}} = S_{a_{j}}\cup ~\mu_{a_{j}}(S_{a_{j}})$ and $\s_{b_{j}} = S_{b_{j}}\cup ~ \mu_{b_{j}}(S_{b_{j}})$ $\forall \ i \in \{1,\ldots , m\}$.
 
 \begin{lem} \label{lemma1}
 For any $\sigma \in I$, let either $a_i \in \sigma$ or $\sigma \cup \{a_i\} \in I$. Then, there exists $a_{j} \leq a_{i}$
 such that $\sigma \in  \s_{a_{j}}$, {\it i.e.}, $\si \in C$  implies that $\si \subseteq B$ and $a \in N(\si)$ $\forall~ a \in A$.
\end{lem}
\begin{proof} 
If $i=1$, then $\sigma \in \s_{a_1}.$  Inductively, let the result hold for all $ l <m$ and $l'\in \{l+1, \ldots m \}$ be the smallest integer such that either $a_{l'} \in \sigma$ or  $\sigma \cup \{a_{l'}\} \in I$. Without loss of generality, let $a_{i} \nin \si$ and $\si \cup \{a_{i}\} \nin I$, $\forall ~ i < l'$.
If $a_{l'} \in \si$, then $\si \sm \{a_{l'}\} \cup \{a_{j}\} \nin I$ $\forall ~j <l'$. Therefore, $\si$, $\si \sm\{ a_{l'}\} \in 
\s_{a_{l'}}.$
 Similarly,  $ \si \cup \{a_{l'}\} \in  I$ implies that  $\si$, $\si \cup \{a_{l'}\} \in \s_{a_{l'}}.$
\end{proof}

Define the set $B_{s}$ to be $\{b_{s}, b_{s+1}, \ldots, b_{s+m-d-1}\}$, where the additions in the subscripts are modulo $m$. $\tau \subseteq B_{s}$ will imply that $b_{s}\in \tau$.
The following is an  easy observation.

\begin{lem}\label{rmk0}
Let $\si \subseteq B_{s}$ for some $s \in \{1,\ldots,m\}$. Then,
\begin{enumerate}
\item $1 <i <s$  and $b_{i} \in B_{s}$ implies that $b_{j} \in B_s$ $\forall ~ j \in \{1,\ldots, i-1\}$.
\item  $s <j \leq m$ and $b_j \in B_{s}$ for some $j \in \{s+1, \ldots, m\}$ implies that $b_{i} \in B_{s} ~\forall ~ i\in \{s+1, \ldots, j-1\}$.
\end{enumerate}
\end{lem}
 
We now determine the necessary and sufficient conditions for $\si \cup \{a\} \in I$, where $a \in A$.

\begin{lem} \label{ind} For any $a \in A$, 
 $\si \cup \{a\} \in I  \iff \si \subseteq B_{s}$  for some $s \in [m]$. 
 \end{lem}

\begin{proof} If $\si \cup \{a\} \in I$, then from lemma \ref{lemma1}, $\si \subseteq B\sm N(a)$. Since $|N(a)|=d$ and $|B_{s}|=m-d$,  we see that $\si \subseteq B_{s}$, for some $b_{s} \in \si$.

Conversely, let  there exist $s\in [m]$ with $\si \subseteq B_{s}$. Observe that $B_{s} \cap N(a_{s+m-d}) =\emptyset$. Therefore, $\si \cup \{a_{s+m-d}\} \in I$.
 \end{proof}
 
 \begin{rmk}\label{rmk1} Let $\si \subseteq B$. Here,
 $\si \in M \iff \si \not\subseteq B_{i}$, $\forall ~ i \in [m]$.
 Therefore, for any $b\in B$, if  $\si \in M$ and $b \nin \si$, then $\si \cup \{b\} \in M$ and if $\si \nin M$ and $b \in \si$, then $\si \sm \{b\} \nin M$.
 \end{rmk}

The following result deals with necessary conditions for $\si$ to be a critical cell.

\begin{lem}\label{lemma2}
If $\si \in C$, then $|\si| \geq 2$ and $\si =\{b_{j_1}, b_{j_2}, \ldots, b_{k}\}$, where $ j_1=1 < j_2 < \ldots < k$.
 \end{lem}
 \begin{proof}
 Since $\si \in M$, using Remark \ref{rmk1}, $\si \not\subset B_{s}$, for any $s\in [m]$. If $b_1 \nin \si$,
 then $\si \cup \{b_1\} \not\subset B_{s}$ $\forall ~ s\in [m]$. Thus, $\si, \si \cup \{b_{1}\} \in \s_{b_{1}}$,
 which implies that $\si$ is not a critical cell.  Thus, $b_{1}\in \si$. If $|\si|=1$, then $\si =\{b_1\} \subseteq B_{1}$, a contradiction.
  \end{proof}

Henceforth, all the cells $\sigma$ considered will satisfy the properties
\begin{equation}\label{eqn0}
 \si \in M \ \text{and} \
 \si = \{b_{j_1}, b_{j_2} ,\ldots, b_k\} \ \text{where} \ \  j_1=1 < j_2 < \ldots < k.\\
 \end{equation}
  
 \begin{lem}\label{lemma6} The following hold for any $\si$ satisfying Equation \ref{eqn0}.
 
 \begin{enumerate} 
 \item If $j_{2} \leq d+1$, $k =j_2+m-d-1$ and $b_{i} \nin \si$ for some $i \in [m] \sm \{1\}$, then 
 \begin{enumerate}
 \item $\{\sigma \cup \{b_i\}\} \setminus \{b_1\} \in M \ \forall \ i > k$ and $i \in \{2, \ldots, j_{2}-1\}$,

 \item $\{\sigma \cup \{b_i\}\} \setminus \{b_{j_2}\} \in M \ \forall \ i \in \{j_2+1, \ldots, d+1\}$.
 \end{enumerate}
 \item Let $s\geq 3$. If $\si \sm \{b_{j_{l}}\} \nin M$ for $l \in \{2, \ldots, s-1\}$, then $j_{l+1} \geq j_{l-1} +d+1$.
 
 \end{enumerate}
 \end{lem}
 
 \begin{proof} Since $\sigma \in M,$ from Remark \ref{rmk1}, $\si \not\subseteq B_{s}, ~\forall ~ s\in [m]$. \hfill(i)
 \begin{enumerate} 
 \item \begin{enumerate}
 \item Assume that $\sigma \cup \{b_i\} \setminus \{b_1\} \subseteq B_{s}$, $ s\neq 1$.  If $i>k$, then $ s>j_2$ (as $k=m+j_2-d-1$). Since, $b_{j_2} \in B_{s}$, using Lemma \ref{rmk0}, $b_1\in B_{s}$,
  {\it i.e.} $\si \subseteq B_{s}$, contradicting (i). 

Since $i+m-d-1 < j_2+m-d-1=k$,  $\forall ~i \in \{2, \dots , j_2-1\}$, $b_{k}\in B_s$ implies that $ s >i>1$. From Lemma \ref{rmk0}, $b_i \in B_s \Rightarrow b_{1} \in B_{s}$, {\it i.e.} $\si \subseteq B_{s}$, a contradiction.  
  
\item Suppose that $\si\cup \{b_{i}\}\setminus \{b_{j_2}\} \subseteq B_{s}$ with $s \neq j_2$. Clearly, $s \neq 1$. If $ s \in \{j_2+1, \dots , d+1\}$, then $s+m-d-1 \leq m$ which implies that $b _{1} \nin B_{s}$. If $s>d+1$, then $b_{1}, ~b_i \in B_s$ implies that $b_{j_2} \in B_{s}$. In both the cases, $\si \nin M$. Therefore, there does not exist any $ s\in [m]$ such that $\sigma \cup \{b_{i}\}\setminus \{b_{j_2}\} \ \subseteq B_{s}$.
 \end{enumerate}
 
 \item Let $\si \sm \{b_{j_{l}}\} \subseteq B_{q}$, $q \neq j_{l}$. If $q \leq j_{l-1}$, then $b_{j_{l+1}} \in B_{q} \Rightarrow b_{j_{l}} \in B_{q}$ and if $q >j_{l+1}$, then $b_{j_{l+1}}\in B_{q} \Rightarrow b_{j_{l}}\in B_{q}$. In both these cases, $\si \nin M$. Therefore, $q=j_{l+1}$.  Here, $j_{l+1} + m -(d+1) ~(\text{mod}~ m) \geq j_{l-1}$ and $j_{l-1} <j_{l+1}$ imply that $j_{l+1} > d+1$ and $j_{l+1} \geq j_{l-1}+d+1$.
 \end{enumerate}
  \end{proof}

 We  now determine the necessary and sufficient conditions for
 $\sigma \notin \s_{b_{i}}$, $\forall ~ i \in \{1, \dots ,d+1\}$.
 
\begin{lem} \label{lemma3} Let $\si$ satisfy $(1)$.  Then,
$\sigma \notin \s_{b_i}\ \forall ~i \in \{1,\dots , d+1\}$ if and only if

\noindent $(1)$ $ j_2 \leq d+1$, ~
$(2)$ $ \sigma \setminus \{b_1\} \nin M,$ ~
$(3)$  $k = j_2 + m-d-1$ ~ and ~
$(4)$  $\sigma \setminus \{b_{j_2}\}\nin M$.
\end{lem}

\begin{proof}$(\Longrightarrow)$ We first assume that 
\begin{equation} \label{eqn2.5}
\sigma \notin \s_{b_i} ~\forall ~i \in \{1,\dots , d+1\}.
\end{equation}

\begin{enumerate}
\item[(1)] If $j_2 \geq d+2$, then $j_{2}+m -d-1 >m$ implies that $\sigma \subseteq
  B_{j_2}$, a contradiction. 
  
\item[(2)] $\si \sm \{b_1\} \in M$ implies that $\sigma \in \s_{b_1},$ a contradiction. Hence, $\sigma \setminus \{b_1\} \nin M$.
   
  
\item[(3)] From $(2)$ and Remark \ref{rmk1}, $\si \sm \{b_1\} \subseteq B_{j_2}$. Thus, $k \leq j_2 + m-d-1 \leq m$. Let $k'= k-m+d+1$ and $k' <j_2$. Since, $k' \leq 1$ implies that $k \leq m+d$, {\it i.e.} $ \si \subseteq B_{1}$, we get $1 <k' <j_2$, $b_{k'} \nin \si$ and $\si \cup \{b_{k'}\}  \sm \{b_{1}\} \subseteq B_{k'}$. Therefore, $\si \cup \{b_{k'}\} \notin \s_{b_1}$.  Since, $d \geq 1$, using equation \ref{eqn2.5}, we get $k' > 2$. 
Let $\si'= \{ \sigma \cup \{b_{k'}, b_{i}\} \} \setminus \{b_1\} \in M$ for $i \in \{2, \ldots , k'-1\}$. 

Let $\si' \subseteq B_{s} $, $s \in [m]\sm \{ 1\}.$ If $s$ is either $i$ or $k'$, then $ b_{k}$ or $b_{i} \nin B_{s}$, respectively and if $s \geq j_{2}$, then $\si \subseteq B_{s}$. Thus $\si' $ and $\si' \cup \{b_{1}\} \in M$. Therefore, 
  $\si \cup \{b_{k'}, b_{i}\} \in  \s_{b_1}$. Here, $\si, \ \si \cup \{b_{k'}\} \in \s_{b_{k'}}$, an impossibility as $k' <j_2 \leq d+1$.
    Hence, $k'=j_2.$
 
\item[(4)] Suppose  $\sigma \setminus \{b_{j_2}\}  \in M$. Here, $j_2 \neq k$  (as $j_2=k \Rightarrow \si \sm \{b_{j_{2}}\} =\{b_{1}\} \subseteq B_{1}$).

Using, $(2)$, we get $\si \sm \{b_{1}, b_{j_2}\} \nin M$, {\it i.e.}, $\sigma \setminus \{b_{j_2}\}  \notin \s_{b_1} $. From the hypothesis, $\si \nin \s_{b_{2}}$. Thus, $j_2 > 2$.  Let  $i \in \{2, \ldots , j_2-1\}$.

 Here, $b_{i} \nin \si$ and $\{\si \sm \{b_{j_2}\}\} \cup \{b_{i}\} \in M$ (since $\sigma \setminus \{b_{j_2}\}  \in M$). Let $\si' = \{\si \cup \{b_{i}\}\} \sm \{b_{1}, b_{j_2}\} $. 
If possible, let $\si' \subseteq B_{s}$ for some $ s \in [m]\sm \{1,j_2\}$. For $1 <s \leq i$,  $i < j_2 $ implies that $i+m-d-1 < j_2+m-d-1 = k$. Here, $b_{k} \nin B_s$. If $s>i$, then from Lemma \ref{rmk0} $(1)$,  $\{\sigma \cup \{b_i\}\} \setminus \{b_{j_2}\} \subseteq B_{s}$ (as $b_{i} \in B_{s}$), a contradiction.  
Thus, $\si' \in M$.
Therefore,  $\{\sigma \cup \{b_i\}\} \setminus \{b_{j_2}\} \in \s_{b_{1}}$ and $\sigma \setminus \{b_{j_2}\} \notin \s_{b_i}$ $\forall ~ i \in \{1, \ldots , j_2-1\}.$ This shows that $\sigma, \si \sm\{b_{j_2}\} \in \s_{b_{j_2}}$, a contradiction.
  Hence, $\sigma \setminus \{b_{j_2}\} \nin M.$ 
 \end{enumerate}

\noindent $(\Longleftarrow)$ Let $(1)$, $(2)$, $(3)$ and $(4)$ hold.  From $(2)$ and $(4)$, we see that
$\sigma \notin \s_{b_1}, \s_{b_{j_2}}.$  

If $j_2 >2$, then for all $i \in \{2, \ldots , j_{2}-1\}$, we have $b_{i} \nin \si$,  $\{\sigma \cup \{b_i \}\} \setminus \{b_1\} \in M$ (by Lemma \ref{lemma6} $1(a)$) and  $\sigma \cup \{b_i\} \in M$ (by Remark \ref{rmk1}). 
Therefore, 
$\sigma \notin \s_{b_{i}} ~\forall ~ i \leq  j_2, ~ j_{2}>1.$

    If $j_2 = d+1$, the result follows.
 Now, let $j_2 < d+1$ and $l \in \{j_2+1, \ldots , d+1\}$.  Observe that $\si \sm\{b_{j_2}\} \nin M$ implies that $j_3 \geq 1+d+1$ (by Lemma \ref{lemma6} $(2)$). Thus, $b_{l} \nin \si$.
 
 If $j_2 = k$, then $m = d+1$, $\si =\{b_1, b_{j_2}\}$ and $B_{s}=\{b_{s}\} \ \forall \ s$. Therefore, $\si \cup\{b_{l}\}$ and 
 $\{\si \cup\{b_{l}\}\}\sm \{b_{1}\} \in M$, thereby showing that $\si \cup\{b_{l}\} \in \s_{b_{1}}$. Thus,  $\si \nin \s_{b_{l}} ~\forall ~ l \leq d+1$.

Now, consider the case when $j_2 \neq k$, {\it i.e.} $j_{2} <k$.  We have the following two cases.
 
 \smallskip
 
 \begin{enumerate}
 \item[(i)] $l > k$.\\ Here, $b_{l} \nin \si$, $\si \cup \{b_{l}\} \in M$ and $\{\sigma \cup \{b_l\} \} \setminus \{b_1\} \in M$
(Lemma \ref{lemma6} $1(a)$).
 Thus, $\sigma \notin \s_{b_l}$.

 \item[(ii)] $l \leq k$.  
 
 $\si^{1}= \si \cup \{b_{l}\} \in M$ (Remark \ref{rmk1}) and $ \si^{1} \sm \{b_1\} \subseteq B_{j_2}$ (as $k =j_2+m-d-1$).
Therefore, $\si^{1}\sm \{b_1, b_{j_2}\} \nin M$ and
 $\sigma^{1}$, $\sigma^{1} \setminus \{b_{j_2}\} \nin \s_{b_1}$. Thus, for  $j_2 = 2$, $\si^{1}, \sigma^{1} \sm \{b_{2}\} \in \s_{b_2}$.  Let $j_{2} \neq 2$ and $i \in \{2, \ldots, j_2-1\}.$ 
 Suppose $\si^{2}=\{\sigma^{1}\cup \{b_i\} \} \setminus \{b_1, b_{j_2}\} \subseteq B_{s}$, $s \neq 1, j_2$. 
If $s \leq i $, then  $b_{k} \nin B_{s}$ (as $k=j_2+m-(d+1)$).  
If $ i < s \leq l$, then $l \leq d+1$ implies that $s+m -(d+1) \leq m$. Here $b_{i} \nin B_{s}$. 
 If $s >l$, then, from Lemma \ref{rmk0}, $b_{1}, b_{j_{2}} \in B_{s}$, {\it i.e.} $\si \subseteq B_{s}$. Therefore, $\{\si \cup \{b_{l}, b_{i}\}\} \sm \{b_{1}, b_{j_{2}}\} \in M$ $\forall \ i \in \{2, \ldots , j_{2}-1\}$ and $l \in \{j_2+1, \ldots , d+1\}$. 
Thus,
 $\sigma\cup \{b_{l},b_i\} $ and  $\{\sigma \cup \{b_{l},b_i\} \} \setminus \{b_{j_2}\}   \in \s_{b_{1}}$.
 We now conclude that $\si \cup \{b_{l}\}$ and $\{\si \cup \{b_{l}\}\}\sm \{b_{j_2}\} \nin \s_{b_i}$ $\forall \ i<j_2$ which shows that $\si \cup \{b_{l}\} \in \s_{b_{j_2}}$. Therefore,  $\sigma \notin \s_{b_l}$ $\forall \ l \leq d+1$.  This completes the proof.
  \end{enumerate}
 \end{proof}

Let $i_1=1$ and $j_{1} \in [m]$, $i_r = 1+ (r-1)(d+1)$ and $j_r = j_1 + (r-1)(d+1)$, $r \in \{2, \ldots , t\}$. For each $l \in \{1,\ldots , t\}$, $\si$ is said to satisfy the property $P_{l}$ if 
 \begin{itemize}
 \item[(a)] $\sigma = \{b_{1}, b_{j_1}, \ldots, b_{i_l}, b_{j_l}, \ldots, b_k\}$, $j _1 \leq d+1$,  $k =m+j_1-d-1$ and
  \item[(b)]  $\sigma \setminus \{b_{i_r}\}$,  $\sigma \setminus \{b_{j_r}\}\nin M$, $\forall \ r \in \{1,\ldots, l\}.$
 \end{itemize}
 
Observe that $1 <j_{1}<i_{2} <j_{2} <\ldots i_{l} < j_{l} \leq \dots   \leq k$. Now, we extend Lemma \ref{lemma3}.

\begin{lem} \label{lemma4} Let $l \in \{1,\ldots , t\}$.
 $\sigma \notin \s_{b_{i}} \ \forall \ i \in \{ 1, \ldots ,l(d+1)\}$ if and only if
 $\sigma$ satisfies $P_{l}$.
\end{lem}
\begin{proof} From Lemma \ref{lemma3}, the result holds for $l=1$.
 Inductively, assume that the result holds for all  $1\leq r \leq l < t$.  For such simplices $\si$, we first prove the following:
 
 \begin{prop} \label{prop2} 
  $\sigma \notin  \s_{b_{i}} \ \forall  \ i \in \{1, \ldots, i_{l+1}\}$ if and only if
~ $(1)$  $b_{i_{l+1}} \in \si$, ~
  $(2)$ $\si \sm \{b_{i_{l+1}}\}\nin M$ ~ and ~
  $(3)$ $\ b _s \nin \si$ $\forall ~ s \in  \{ x \in [m] \sm \{i_{l+1}\} \ | \ j_{l} < x < j_{l+1}\}$.
  \end{prop}
  
 \noindent {\it Proof :} ${\emph (\Longrightarrow)}$ Let $\sigma \notin  \s_{b_{i}}, $ for  $ i \in \{1, \ldots , i_{l+1}\}$.  Since $\si$ satisfies $P_{l}$, $j_1 \leq d+1$ and $k=j_1+m-(d+1)$. Further, $l <t$ implies that $j_{l}< i_{l+1} <k$.
  \begin{enumerate}
\item[(1)]  Let $p\in \{j_{l}+1, \ldots, k\}$ be the smallest integer such that $b_p \in \si$. Since $\sigma \setminus \{b_{j_l}\} \nin M$, $p \geq i_{l} +d+1 =i_{l+1}$ (Lemma \ref{lemma6} $(2)$).
 If possible, let $p >i_{l+1}$.
 
 Here, $b_{i_{l+1}} \notin \sigma$ as $j_{l}< i_{l+1}$. Observe that $j_r +m-d-1 \ (\text{mod}\ m) = j_{r-1}$  and $i_{r+1} +m-d-1 \ (\text{mod}\ m) = i_r$ for $2 \leq r \leq l$. Therefore,  
   
 \begin{equation}\label{eq1}
  \si \cup \{b_{i_{l+1}}\} \sm \{b_{s}\} \subseteq \begin{cases}
    B_{j_{r}}, &  ~~s=i_{r}~~r \in \{1,\ldots, l\}  \\
   B_{i_{r+1}}, &  ~~s=j_{r}~~r \in \{1,\ldots, l-1\}  \\
   B_{i_{l+1}}, &  ~~s=j_{l}. \\
\end{cases}
\end{equation}

\medskip

  $\si \cup \{b_{i_{l+1}}\}$ satisfies $P_{l}$. Therefore, $\sigma  \cup \{b_{i_{l+1}}\} \notin \s_{b_i} \ \forall \ i \in \{1, \ldots, l(d+1)\}$, by the induction argument for $\si \cup \{b_{i_{l+1}}\}$. Since, $\sigma \notin \s_{b_i}$ $\forall ~ i \leq l(d+1)$, we see that $\sigma \in \s_{b_{i_{l+1}}}$, a contradiction. 
Therefore, $b_{1+l(d+1)} \in \sigma$. 

  \item[(2)]  
  $\si \sm \{b_{i_{l+1}}\}$ satisfies both properties of $P_{l}$. Therefore, $\si \sm \{b_{i_{l+1}}\}\nin \s_{b_j}  \ \forall \ j \in \{1, \ldots, l(d+1)\}$. Since, $\si \nin \s_{b_{i_{l+1}}}$, we see that $\si \sm \{b_{i_{l+1}}\} \nin M$.

\item[(3)]
   Let $q$ be the smallest integer in $\{i_{l+1} +1, \ldots,  k\}$ such that $b_{q} \in \sigma$.  
   Since $\si \sm \{b_{i_{l+1}}\} \nin M$, from Lemma \ref{lemma6} $(2)$, $q \geq j_{l}+d+1$.  As $j_{l+1} -j_{l}= d+1$, $q \geq j_{l+1}$. This proves $(3)$.
   \end{enumerate}
   
\noindent ${\emph (\Longleftarrow)}$  Conversely, let $(1)$, $(2)$ and $(3)$ of Proposition \ref{prop2} hold.

 $\sigma \setminus \{b_{i_{l+1}}\}\nin M$ implies that $\sigma \notin \s_{b_{i_{l+1}}}$.  Since
   $\sigma \notin \s_{b_j} ~\forall \ j \in \{1, \ldots, l(d+1)\}$ by the hypothesis, the result follows, thereby proving Proposition \ref{prop2}.
To complete the induction step, we show that $\si \nin \s_{b_{i}}$, $\forall ~ i \leq (l+1)(d+1)$ if and only if it satisfies $P_{l+1}$.

 \medskip
 
 \noindent ${\emph (\Longrightarrow)}$  Let $\sigma \notin  \s_{b_j}$, $\forall ~ 1 \leq j \leq (l+1)(d+1)$ and $b_{q} \in \si$ for a least element $q$ in $\{1+i_{l+1}, \ldots , k\}$. Since $\si \sm \{b_{i_{l+1}}\} \nin M$ (Proposition \ref{prop2}), from Lemma \ref{lemma6}$(2)$, $q \geq j_{l} +d+1$.  We now prove:

\begin{enumerate}
\item[(i)] $q=j_{1}+l(d+1)$ {\it i.e.}, $b_{j_{l+1}} \in \si$.

If $q >j_{l+1}$, then $b_{j_{l+1}}\nin \si$. Let $\si'= \si \cup \ \{b_{j_{l+1}}\}$. For each $b_{s} \in \si$, $s \leq j_{l}$,
      $\si' \sm \{b_{s}\}$ satisfies equation \ref{eq1} and  $\si' \sm \{b_{s}\} \nin M$.  Since $q +m-(d+1) ~(\text{mod}~m) >j_{1}+(l-1)(d+1)=j_{l}$, we see that $\si ' \sm \{b_{i_{l+1}}\} \subseteq B_{j_{l+1}}$, thereby satisfying $(2)$. Since
Proposition \ref{prop2} $(3)$ is true by the construction,   
 $\si \cup \{b_{j_{l+1}}\} \notin \s_{b_i} ~ \forall ~i \in \{1, \ldots , i_{l+1}\}$ (Proposition \ref{prop2}). 

For $i \in \{i_{l+1} +1, \ldots , j_{l+1}-1\}$, $b_{i} \notin \si \cup \{b_{j_{l+1}}\} $ and $\si \cup \{b_{j_{l+1}}, b_{i}\} \in M$. Further, $\si \cup \{b_{j_{l+1}}, b_{i}\} \sm \{b_{s}\}$ satisfies equation \ref{eq1}, $\forall \ s \leq j_{l}$. 
Therefore, $\si \cup \{b_{j_{l+1}}, b_{i}\}$ satisfies $P_{l}$ and by the induction argument $\si \cup \{b_{j_{l+1}}, b_{i}\} \nin \s_{b_j}$ $\forall \  j \leq l(d+1)$. 
Since $\si \cup \{b_{j_{l+1}}, b_{i}\} \sm \{b_{i_{l+1}}\} \in M$ (by Lemma \ref{lemma6}$(2)$), using Proposition \ref{prop2}, we get  $\si \cup \{b_{j_{l+1}}, b_{i}\} \sm \{b_{i_{l+1}}\}\in \s_{b_{i_{l+1}}}.$ Thus, $\si\cup\{b_{j_{l+1}}\} \nin \s_{b_j}$, $\forall \ j <j_{l+1}$, which implies that $\si \in \s_{b_{j_{l+1}}}$. But, $\si \nin \s_{b_{i}} \ \forall \ i \leq (l+1)(d+1)$. Since $j_1 +l(d+1) \leq (l+1)(d+1)$, this is impossible. Hence, $q=j_{1}+ l(d+1)$.

\item[(ii)] $\si \sm \{b_{j_{l+1}}\} \nin M$.  Here, either $j_{l+1} =k $ or $j_{l+1}<k$.
\begin{enumerate}
\item[(a)] Let $j_{l+1}=k$. Since $1+m-d-1  \geq 1+(t-1)(d+1) \geq i_{l+1}$, we see that $\si \sm \{b_{j_{l+1}}\} \subseteq B_{1}$.

\item[(b)] Let $j_{l+1}<k$. Suppose that $\si \sm \{b_{j_{l+1}}\} \in M$.

 For each $ s \leq i_{l+1}$, $\si \sm \{b_{s}\} \nin M \Rightarrow \si \sm \{b_{s}, b_{j_{l+1}}\} \nin M$. $\si \sm \{b_{j_{l+1}}\}$ satisfies $P_{l}$ (since $j_{l+1} <k$) and also the three conditions in Proposition \ref{prop2}. Proposition \ref{prop2} therefore shows that $\si \sm \{b_{j_{l+1}}\}\nin \s_{b_j}$ $\forall \ j \leq i_{l+1}$.
If $j_1=2$, then $j_{l+1}= 1+i_{l+1}$ and $\si \in \s_{b_{j_{l+1}}},$ an impossibility as $j_{l+1} \leq (l+1)(d+1)$.
Thus, $j_{l+1} > 1+i_{l+1}$.

  For each fixed $i \in \{i_{l+1}+1, \ldots, j_{l+1} -1\}$, consider $\tau_{i} = \si \cup \{b_{i}\} \sm \{b_{j_{l+1}}\}$. If $\tau_{i} \in M$, then $\tau_{i} \sm \{b_{i_{l+1}}\} \in M$ (from Lemma \ref{lemma6}$(2)$ as $ i \ngeq j_{l}+d+1=j_{l+1}$). Therefore, $\exists  \ j \leq i_{l+1}$ such that $\tau_{i} \in \s_{b_{j}}$ (Proposition \ref{prop2}). Hence, $\si \sm \{b_{j_{l+1}}\} \nin \s_{b_{j}}, \ \forall \ j <j_{l+1}$ and $\si, \ \si \sm \{b_{j_{l+1}}\} \in \s_{b_{j_{l+1}}}$, which is not possible. If $\tau_{i} \nin M$, then the previous statement holds, giving a contradiction. Therefore, $\si \sm \{b_{j_{l+1}}\} \nin M$.
\end{enumerate}
  \end{enumerate}

 \noindent ${\emph (\Longleftarrow)}$  Assume that $\si$ satisfies $P_{l+1}$ and $\si \nin \s_{b_i}$ $\forall ~ i \leq l(d+1)$. 
 Since $\si$ satisfies $P_{l+1}$, $\si \sm \{b_{i_{l+1}}\}, ~\si \sm \{b_{j_{l+1}}\}\nin  M$. Thus, $\si \nin \s_{b_i}$ for $i = i_{l+1}, ~j_{l+1}$.
  
 For any $i\in [m]$ such that $i_{l+1} < i <j_{l+1}$, we see that $b_{i} \nin \si$ and $\si \cup \{b_i\} \in M$. Since $j_{l}+ (d+1) =j_{l+1} >i$, from Lemma \ref{lemma6}$(2)$, $\si \cup \{b_i\} \sm \{b_{i_{l+1}}\} \in M$. Clearly, $\si \cup \{b_{i}\}$ satisfies $P_{l}$ and therefore, from Proposition \ref{prop2}, $\si \cup \{b_{i}\} \in \s_{b_{i_{l+1}}}$. Hence, $\si \nin \s_{b_i}$, whenever $ i \leq j_{l+1}$.  We now consider the following two cases.
 
 \begin{enumerate}
 \item[(1)] $j_{1}=d+1$.
 Here, $j_{l+1}=(l+1)(d+1)$ and hence the result follows.
 
\item[(2)] $j_{1}<d+1$. Consider $i \in \{j_{l+1}+1, \ldots, (l+1)(d+1)\}$.

\begin{enumerate}
 \item[(2.1)] Let  $i>k$. Here, $b_{i} \nin \si$ and from Lemma \ref{lemma6}$(1)$, $\si \cup\{b_i\} \sm \{b_{1}\} \in M$, thereby showing that $\si \cup\{b_i\} \in \s_{b_1}$. Hence, $\si \nin \s_{b_i}$ $\forall ~ i \leq (l+1)(d+1)$.
 
 \item[(2.2)] Let $i \leq k$. If $b_{i} \in \si$, then by Lemma \ref{lemma6}$(2)$, $\si \sm \{b_{j_{l+1}}\} \nin M \Rightarrow i \geq i_{l+1}+d+1 =1+(l+1)(d+1)$, a contradiction. Therefore, $b_{i} \nin \si$.   
\begin{claim}\label{claim3}
$\si \cup \{b_{i}\}, ~ \si \cup \{b_{i}\} \sm \{b_{j_{l+1}}\}  \in \s_{b_{j_{l+1}}}$.
\end{claim}
 
 Suppose  $\si \cup \{b_{i}\} \sm \{b_{j_{l+1}}\} \subseteq B_{q}$. If either $q<i$ or $q>i$, then $b_{i} \in B_{q} \Rightarrow b_{j_{l+1}} \in B_{q}$, which implies that $\si \subseteq B_{q}$. From Lemma \ref{lemma6}(2) we get $q\neq i$. Therefore, $\si \cup \{b_{i}\} \sm \{b_{j_{l+1}}\} \in M$.
Using equation \ref{eq1}, we see that $\si \cup \{b_{i}\} \sm \{b_{s}\} $ and $\si \cup \{b_{i}\} \sm \{b_{s}, b_{j_{l+1}}\}  \nin M ~ \forall ~ s \leq i_{l+1}$. From Proposition \ref{prop2}, we conclude that $\si \cup \{b_{i}\}$ and $\si \cup \{b_{i}\} \sm \{ b_{j_{l+1}}\} \nin \s_{b_i}$ $\forall ~ i \leq i_{l+1}$. 

If $j_1=2$, then $j_{l+1}= 1+ i_{l+1}$. Here, $\si \cup \{b_{i}\} \sm \{b_{j_{l+1}}\}  \in \s_{b_{j_{l+1}}}.$
Let $j_1 >2$ and $s \in \{1+i_{l+1}, \ldots , j_{l+1}-1\}$.

Using Lemma \ref{lemma6}, we see that $\si \cup \{b_i, b_s\} \sm \{b_{i_{l+1}}\} \nin M$ implies that $s \geq j_{l}+d+1=j_{l+1}$ and $\si \cup \{b_i, b_s\} \sm \{b_{i_{l+1}}, b_{j_{l+1}}\} \nin M$ implies that $i \geq s+d+1\geq 1+l(d+1)+(d+1)$, both of which are impossible.  
Thus, $\si \cup \{b_i, b_s\}$ and $\si \cup \{b_i, b_s\} \sm \{b_{i_{l+1}}\}$ satisfy $P_{l}$. From Proposition \ref{prop2},  $\si \cup \{b_i, b_s\}$ and $\si \cup \{b_i, b_s\} \sm \{b_{j_{l+1}}\} \in \s_{b_{i_{l+1}}}$, {\it i.e.} $\si \cup \{b_{i}\}, \si \cup \{b_i\} \sm \{b_{j_{l+1}}\}  \nin \s_{b_{s}}$ $\forall ~ s <j_{l+1}$, thereby proving the claim.

From Claim \ref{claim3}, $\si \cup \{b_{i}\} \in \s_{b_{j_{l+1}}}$ $\forall ~ i \in \{j_{l+1}+1, \ldots, (l+1)(d+1)\}$. Therefore, $\si \nin \s_{b_i} ~\forall ~ i \leq (l+1)(d+1)$.  
\end{enumerate}
\end{enumerate}
\end{proof}

We now discuss the criteria for cells satisfying the property $P_{t}$ to be critical cells.

 \begin{lem}\label{lemma8} Let $\si$ satisfy $P_{t}$. Then, $\si \in C$  if and only if either
$j_{t}=k$ or
 $\si =\{b_{1}, b_{j_{1}}, b_{i_{2}}, b_{j_{2}}, \ldots, b_{i_{t}}, b_{j_{t}}, b_{k}=b_{i_{t+1}}\}$.
 \end{lem}

 \begin{proof}
 Since $\si$ satisfies property $P_{t}$, $\si \nin \s_{b_{i}}$ for all $i \in \{1, \ldots, t(d+1)\}$ and $\si =\{b_{1}, b_{j_1},  \ldots, b_{i_{t}}, b_{j_{t}}, \ldots, b_{k}\}$, where $j_{1} \leq d+1$ and $k=j_1+m-(d+1)$. Since $m=t(d+1) +\alpha$, we see that either $\alpha =0$ or $\alpha >0$.
 
 \begin{enumerate}
 \item $\alpha=0$.
 
 Since $j_{t} = j_{1}+(t-1)(d+1) = j_1+m-(d+1)=k$, we observe that $\si$ is of the form $\{b_{1}, b_{j_1},\ldots, b_{i_{t}}, b_{j_{t}}=b_{k}\}$.  Here $t(d+1)=m$ implies that $\si$ is a critical cell. The converse is straightforward.
 
 \item $\alpha>0$.
 
  Here, $j_{t} <k$. Let 
   $p \in \{j_t+1, \ldots, k\}$ be the smallest integer such that $b_p \in \si$.
  Since $\si \sm \{b_{j_{t} }\} \nin M$ (property $P_{t}$) and $\si \nin \s_{b_{i}} \ \forall \ i \leq t$, by an argument similar to that in $(1)$ of  Proposition \ref{prop2}, we get $\si \sm \{b_{j_{t}}\} \subseteq B_{p}$ and $ p = 1+t(d+1)=i_{t+1}$.
Thus, if $\si$ satifies $P_{t}$, then $\si =\{b_{1}, b_{j_{1}}, \ldots, b_{i_{t}}, b_{j_{t}}, b_{i_{t+1}}, \ldots , b_{k}\}$.

\begin{enumerate}
\item $\si$ is a critical cell implies that $i_{t+1}=k$.
If $i_{t+1} <k$, then $j_{1} > d-\alpha +2$. 

Suppose that there exists $s \in [m]\sm \{i_{t+1}\}$, such that $\si \sm \{b_{i_{t+1}}\} \subseteq B_{s}$. Observe that $k <j_1+t(d+1)$. Thus, $s<k$ implies that $s<j_{1}+t(d+1)$. 

If $s > i_{t+1}$, then $s+m-(d+1)< j_{t}$, showing that $j_t \nin B_{s}$.

If $s < i_{t+1}$, then by Lemma \ref{rmk0}, $b_{k}\in B_{s}$ shows that $b_{i_{t+1}}\in B_{s}$, {\it i.e.} $\si \subseteq B_{s}$. Therefore,  $\si \sm \{b_{i_{t+1}}\} \in M$.  Further, $\si \sm \{b_s, b_{i_{t+1}}\} \nin M$, whenever $\si \sm \{b_{s}\} \nin M.$ Thus, $\si \sm \{b_{i_{t+1}}\} $ satisfies $P_{t}$. From Lemma \ref{lemma4}, 
  $\si \setminus \{ b_{i_{t+1}} \} \notin \s_{b_j}$ for all $ 1 \leq j \leq t(d+1)$. 
Therefore, $\si \in \s_{b_{i_{t+1}}}$, which is impossible as $\si$ is a critical cell. Thus, $k=i_{t+1}$.
  
  \item $i_{t+1}=k$ implies that $\si$ is a critical cell.
  
  Let $t_0 <k$ be the  largest number such that $b_{t_0} \in \si \sm \{b_{k}\}$. Since $t_0 \neq j_{1}+m -(d+1)$, $\si \sm \{b_{k}\}$ does not satisfy $P_{t}$. Thus, either $\si \sm \{b_{k}\} \nin M$ or $\si \sm \{b_{k}\}\in \s_{b_{i}}$, for some $i \leq t(d+1)$ (by Lemma \ref{lemma4}). In both the cases, $\si \nin \s_{b_{k}}$.
  
  If $\alpha=1$, then $m=t(d+1)+1 =k$. Here, $\si$ is a critical cell.
  
  If $\alpha >1$, then for each $i \in \{k+1, \ldots , m\}$, $b_{i} \nin \si$ and $\si \cup \{b_{i}\} \in M$. By the same argument as the one above, $\si \cup \{b_{i}\}$ does not satisfy $P_{t}$. By Lemma \ref{lemma4}, it belongs to $\s_{b_{j}}$, for some $j \leq t(d+1)$.
 Thus, $\si \nin \s_{b_{j}}$, for all $j \leq m$, thereby showing that $\si$ is a critical cell.
 \end{enumerate}
 \end{enumerate}
 
 \end{proof}
 
 We now prove the main result of this paper.

\noindent{\bf Proof of the Theorem \ref{maintheorem}}  
 \begin{proof}   Let $\eta$ be a critical cell. From Lemma \ref{lemma2}, $|\eta| \geq 2$ , $\eta \subseteq B$ and $\eta=\{b_{1}, b_{j_2}, \ldots , b_{k}\}$, where $1 < j_{2} < \dots <k$ and $k \leq m$. Since $m=t(d+1) +\alpha$ and $\eta \nin \s_{b_{i}} \ \forall \ i \leq m$, in particular $\eta \nin \s_{b_{i}}$ $\forall \ i \leq t(d+1)$. Therefore, $\eta$ satisfies $P_{t}$  (by Lemma \ref{lemma4}). Lemma \ref{lemma8} gives the possible critical cells.

\begin{itemize}
 \item[(i)] $\alpha = 0$. 
 	
 	From Lemma \ref{lemma8}, $\eta= \{b_{1}, b_{j_1},\ldots, b_{i_{t}}, b_{j_{t}}=b_{k}\}$. Since $ j_{1} \in \{2, \ldots , d+1\}$,  we have exactly $d$ distinct critical cells of dimension $2t-1$  corresponding to the matching $\mu$ on the poset $P$. From Theorem \ref{dmt1}, $\text{Ind}(G_{m}^{d} )\simeq \bigvee\limits_{\text{d-copies}}S^{2t-1} $.

\item[(ii)] $\alpha > 0.$ 
 
If $d=m$, then $t=0$ and $\alpha=d$. Here, $G$ is a complete bipartite graph. Therefore the maximal  simplices are $\{a_1, a_2, \ldots, a_{m}\}$ and $\{b_1, b_2, \dots, b_{m}\}$, both of which are contractible. Thus, $\text{Ind}(G_{m}^{d} )\simeq S^{0}$.  

Let $t>0$.
 From Lemma \ref{lemma8}, $\eta= \{b_{1}, b_{j_1},\ldots, b_{i_{t}}, b_{j_{t}}, b_{i_{t+1}}=b_{k}\}$. Since $k=i_{t+1}$, we get $j_{1}=d-\alpha +2$. Thus, there exists exactly one critical cell, which is of dimension $2t$. Therefore, from Theorem \ref{dmt1}, $\text{Ind}(G_{m}^{d} )\simeq S^{2t} $.  
\end{itemize}
This completes the proof of Theorem \ref{maintheorem}.
\end{proof}

\bibliographystyle{plain}

\end{document}